\documentclass{amsart}
\usepackage{amsmath,amsthm, amssymb}
\usepackage[hideerrors]{xcolor}
\usepackage{pgf,tikz}
\usetikzlibrary{decorations.pathreplacing}
\usetikzlibrary{arrows}
\usetikzlibrary[patterns]
\usetikzlibrary{shapes,backgrounds}
\usepackage{graphicx}
\usepackage{hyperref}
\usepackage[capitalise, noabbrev]{cleveref}
	\crefname{subsection}{Subsection}{Subsections}
	\Crefname{subsection}{Subsection}{Subsections}
\usepackage{array}
\usepackage{rotating}
\usepackage{longtable}
\usepackage{epstopdf}
\usepackage[T1]{fontenc}
\usepackage{bold-extra}
\usepackage{mathrsfs}
\usepackage[foot]{amsaddr}

\newcommand{\snort}{\textsc{Snort }}

\newcommand{\domineering}{\textsc{Domineering }}
\newcommand{\Snort}{\textsc{Snort}}

\newcommand{\Domineering}{\textsc{Domineering}}

\DeclareMathOperator{\ver}{vert}
\DeclareMathOperator{\obl}{obl}
\newcommand{\incomp}{\not\gtrless}

\theoremstyle{definition} \newtheorem{definition}{Definition}
\theoremstyle{plain} \newtheorem{theorem}[definition]{Theorem}
\theoremstyle{plain} \newtheorem{corollary}[definition]{Corollary}
\theoremstyle{plain} \newtheorem{proposition}[definition]{Proposition}
\theoremstyle{plain} \newtheorem{lemma}[definition]{Lemma}
\theoremstyle{plain} \newtheorem{conjecture}[definition]{Conjecture}
\theoremstyle{definition} \newtheorem{example}[definition]{Example}
\theoremstyle{remark} 
\theoremstyle{definition} 
\theoremstyle{plain} \newtheorem{fact}[definition]{Fact}
\theoremstyle{remark} 
\theoremstyle{definition} 

\usepackage{xspace}
\makeatletter
\DeclareRobustCommand\onedot{\futurelet\@let@token\@onedot}
\def\@onedot{\ifx\@let@token.\else.\null\fi\xspace}
 
\def\ie{{i.e}\onedot}

\makeatother

\newcommand{\Ldomino}[2]{\filldraw[fill=black!30!white] (#1+0.2,#2+0.2)--(#1+0.8,#2+0.2)--(#1+0.8,#2+1.8)--(#1+0.2,#2+1.8)--cycle;}

\newcommand{\Ltoken}[2]{\filldraw[fill=blue!30] (#1+0.5,#2+0.5) circle (0.4); \node at (#1+0.5,#2+0.5) {$L$};}
\newcommand{\Rtoken}[2]{\filldraw[fill=red!30] (#1+0.5,#2+0.5) circle (0.4);\node at (#1+0.5,#2+0.5) {$R$};}

\begin{document}
\title{Bounding Game Temperature using Confusion Intervals}
\author{Svenja Huntemann\textsuperscript{1}\textsuperscript{2}}
\address{\textsuperscript{1} School of Mathematics and Statistics\\ Carleton University\\ Ottawa, Canada}

\author{Richard J.~Nowakowski\textsuperscript{2}}
\address{\textsuperscript{2} Dept.~of Mathematics and Statistics\\
Dalhousie University\\ Halifax, Canada}

\author{Carlos Pereira dos Santos\textsuperscript{3}}
\address{\textsuperscript{3} Center for Functional Analysis, Linear Structures and Applications\\ University of Lisbon \& ISEL--IPL\\
Lisbon, Portugal}


\thanks{The first author's research was supported by the Natural Sciences and Engineering Research Council of Canada (funding reference numbers PDF-516619-2018 and CGSD3-459150-2014) and the Killam Trust. The second author's research was supported by the Natural Sciences and Engineering Research Council of Canada (funding reference number 4139-2014). The third author is a CEAFEL member and has the support of UID/MAT/04721/2019 strategic project.}

\begin{abstract}
For a combinatorial games, temperature is a measure of the volatility, that is, by how much the advantage can change. Typically, the temperature has been measured for individual positions within specific games. In this paper, we give the first general upper bounds on the temperature for any class of games.

For a position $G$, the closure of the set of numbers $\{g\}$ such that  $G-g$ is a first player win, is called the confusion interval of $G$. Let $\ell(G)$ be the length of this interval. Our first main result is: For a class of games $\mathscr{S}$, if there are constants $J$ and $K$ such that $\ell(G^L),\ell(G^R)\leq J$ and $\ell(G)\leq K$ for for every $G\in \mathscr{S}$, 
then the temperature of every game is bounded by $K/2+J$. 
We give an example to show that this bound is tight.

Our second main result is a method to find a bound for the confusion intervals. 
In $G^L-G$ when Left gets to go first, the number of passing moves required by Right to win gives an upper bound on $\ell(G)$.

This is the first general upper bound on temperature.  
As examples of the bound and the method, we give upper bounds on the temperature of subclasses of \domineering and \Snort.
\end{abstract}

\maketitle


\section{Introduction}

Combinatorial games are two-player games of pure strategy such as \textsc{Chess} or \textsc{Go}. In this paper, we consider the temperature of combinatorial games which is a characteristic that can be used to choose a good move. However, it is difficult to calculate. One reason for this is the possibility of making large threats during game play, which increases temperature. Thus it does not monotonically decrease while playing. 

Temperature has been considered in several different contexts (see for example explicitly in \cite{Berlekamp1996,Cazenave2015,Mesdal2009,MullerES2004,NowakowskiS2007} and implicitly in \cite{Berlekamp1988,Wolfe1993}). Being able to bound temperature based on features such as board size will make analysis simpler, particularly when in a sum. However, there is no general applicable theory that can be used to approximate or bound temperature. For a class of games $S$, the boiling point $BP(S)$ is the maximum temperature of any member of $S$. In this paper, we prove an upper bound on the boiling point of a set of games based on the maximum length of the confusion intervals. This is the first known result which holds for all short games. Although this bound is still far from some conjectured bounds and computational evidence, it is optimal in the sense that there are examples of classes in which it is tight. 

In \cref{sec:background} we give all background required from combinatorial game theory for the temperature of a game and discuss a few known results. Our main results are: \cref{thm:TempBound}, which states that the temperature of a game $G$ can be bounded by the length of the confusion interval of $G$ and its options; and \cref{thm:BPbound1}, which shows that therefore the boiling point of a class of games is bounded if the length of the confusion interval of all members of the class is bounded. \cref{sec:BoilingPointBound} is dedicated to proving these two results. As examples of how to apply this new bound, we then discuss the temperatures of \domineering and \Snort. In \cref{sec:domineering} we give an upper bound on the temperature of \domineering snakes fitting within a $2\times n$ grid. And in \cref{sec:Snort} we give an upper bound on the temperature when playing \snort on a path, as well as a general conjecture stating that the temperature is at most the degree of the board one is playing on.

\section{Background}\label{sec:background}

We will begin by introducing concepts from combinatorial game theory required throughout. For more information and proofs of statements in this section, see for example ``Combinatorial Game Theory'' by Siegel \cite{Siegel2013}.

Abstractly, a combinatorial game is a directed graph where the arcs are coloured either blue or red and a token is one one of the nodes. There are two players, Left and Right. Left can move the token along a blue arc and Right along a red arc. The nodes are called \textit{positions} or (in accordance with the colloquial use of the term) games. There is perfect information so the players know the graph and which node the token occupies.  The combinatorial game is  \textit{short} if the digraph is finite and acyclic.

Although correct, there is a better description that allows induction techniques to be used. For a position $P$, let $P^\mathcal{L}$ ($P^\mathcal{R}$) be the set of nodes adjacent via a blue (red) edge. These are the set of \textit{Left options} (\textit{Right options}).  A single option will be denoted by $P^L$ or $P^R$.
A position then can be uniquely described by its options, $P=\{P^\mathcal{L}\mid P^\mathcal{L}\}$. 

Note that for any $P$, the sets $P^\mathcal{L}$ and $P^\mathcal{R}$ need not be disjoint but a game $G$ for which $P^\mathcal{L}=P^\mathcal{R}$ for all positions $P$ is called \textit{impartial}.

In this paper, we only consider short games under \textit{normal play}, meaning that the first player unable to move loses. We denote a combinatorial game by its name in \textsc{Small Caps}. Examples we will be using throughout are the following games.

\begin{definition}
In \domineering (see \cite{Berlekamp1988,LachmannMR2002}), which is played on grids, both players place dominoes. Left may only place vertically, and Right only horizontally. 

In \snort (see \cite{BerlekampCG2004}), which is played on graphs, players place a piece on a vertex which is not adjacent to a vertex containing a piece from their opponent.
\end{definition}

\begin{definition}
The \textit{disjunctive sum} $G_1+G_2$ of two games $G_1$ and $G_2$ is defined recursively as the game in which at each step the current player can decide to move in either game, but not both. Formally,
\begin{align*}
G_1+G_2=\{G_1^\mathcal{L}+G_2,G_1+G_2^\mathcal{L}\mid G_1^\mathcal{R}+G_2, G_1+G_2^\mathcal{R}\},
\end{align*}
where $G_i^\mathcal{L}+G_j=\{G_i^L+G_j:G_i^L\in G_i^\mathcal{L}\}$ and $G_i^\mathcal{R}+G_j=\{G_i^R+G_j:G_i^R\in G_i^\mathcal{R}\}$.
\end{definition}

Finding the outcome of a combinatorial game $G$ is one of the most important goals of any analysis. 

\begin{definition}[Outcome classes]\label{def:OutcomeClasses}
The \textit{outcome classes} are:
\begin{itemize}
\item $\mathscr{N} = \{G:$  the first (next) player can force a win in $G$ \};
\item $\mathscr{P}= \{G:$  the second (previous) player can force a win in $G$ \};
\item $\mathscr{L} =  \{G:$   Left can force a win, no matter who plays first in $G$ \};
\item $\mathscr{R} = \{G:$  Right can force a win, no matter who plays first in $G$ \}.
\end{itemize}
The \textit{outcome} of $G$ is  $o(G)=\mathscr{L}$, $o(G)=\mathscr{N}$, $o(G)=\mathscr{P}$, and $o(G)=\mathscr{R}$ when $G\in \mathscr{L}$, $G\in \mathscr{N}$, $G\in \mathscr{P}$, and $G\in \mathscr{R}$.
\end{definition}
A game whose outcome is in $\mathscr{N}$, that is one in which the first player always has a good move, is also called a first-player win. Similarly, games whose outcomes are in the other classes are called second-player win, Left win, or Right win, respectively.

Convention in combinatorial game theory is to order games by how favourable they are to Left. Games in $\mathscr{L}$ are the most favourable as she can always force a win. There is no differentiation between games in $\mathscr{N}$ and $\mathscr{P}$ since in both cases she in some sense wins half the time, while games in $\mathscr{R}$ are the least favourable as she always loses. Thus we have the partial order on the outcome classes as below.

\begin{center}
\begin{tikzpicture}
	\node (L) at (0,0) {$\mathscr{L}$};
	\node (N) at (-1,-1) {$\mathscr{N}$}
		edge (L);
	\node (P) at (1,-1) {$\mathscr{P}$}
		edge (L);
	\node (R) at (0,-2) {$\mathscr{R}$}
		edge (N)
		edge (P);
\end{tikzpicture}
\end{center}

\begin{example}
As examples for the outcome classes, consider the \domineering positions below, assuming normal play.
\begin{center}
\begin{tikzpicture}[scale=0.5]
	\foreach \x in {1,2,3}{
		\draw (\x,0)--(\x,2);}
	\foreach \y in {0,1,2}{
		\draw (1,\y)--(3,\y);}
	\foreach \x in {6,7}{
		\draw (\x,0)--(\x,2);
		\draw (\x+2,0)--(\x+2,1);}
	\foreach \y in {0,1}{
		\draw (6,\y)--(9,\y);}
	\draw (6,2)--(7,2);
	\foreach \y in {0,1}{
		\draw (12,\y+2)--(14,\y+2);
		\draw (12,\y)--(13,\y);}
	\foreach \x in {12,13}{
		\draw (\x,0)--(\x,3);}
	\draw (14,2)--(14,3);
	\draw (17,0)--(17,1)--(18,1)--(18,0)--(17,0);
	\node at (2,-1) {$\mathscr{N}$};
	\node at (7.5,-1) {$\mathscr{R}$};
	\node at (13,-1) {$\mathscr{L}$};
	\node at (17.5,-1) {$\mathscr{P}$};
\end{tikzpicture}
\end{center}

In the first position from the left, once either player has placed a domino, the other cannot place theirs, thus the first player to go wins. In the second position, Right going first can play in the two bottom left spaces, which leaves no moves for Left, while if Left goes first, she only has one move, leaving another move for Right. Thus Right wins this game, no matter if going first or second. The third position is similarly a Left win. In the fourth position, neither player can move. Thus no matter who goes first, they will lose, implying that the second player to go wins.
\end{example}

Given a fixed winning condition, we say two games $G_1$ and $G_2$ are \textit{equal} and write $G_1=G_2$ if $o(G_1+H)=o(G_2+H)$ for all games $H$. This relation is an equivalence relation. The equivalence class of a game $G$ under ``$=$'', often referred to by its canonical form\footnote{The game with the game tree of smallest depth and of those the one with fewest nodes. This is unique \cite{Siegel2013}}., is called its \textit{game value}.

For a game $G$, we say that the \textit{negative} $-G$ is the game recursively defined as
	\[-G=\{-G^\mathcal{R}\mid -G^\mathcal{L}\},\]
\ie the game in which the roles of Left and Right are reversed. For example, in \domineering this is equivalent to rotating the board by $90^\circ$, and in \snort this is equivalent to switching colours. We use $G-H$ as shorthand for $G+(-H)$.

Similar to equality, we can also define inequalities: We say that $G_1>G_2$ if $o(G_1+H)> o(G_2+H)$ for all games $H$, with the partial order on the outcome classes as discussed above. Two games are incomparable, denoted $G_1\incomp G_2$, if their outcome classes are incomparable, \ie if one is a first-player win and the other a second-player win. Similarly defined are $G_1\geq G_2$ and $G_1\leq G_2$.

Under normal play, we are able to determine the relationship between two games using the following fact by simply determining the outcome class of their difference.

\begin{fact}[{\cite[Section II.1]{Siegel2013}}]\label{thm:DifferenceOutcomes}
Under normal play, for two games $G$ and $H$, we have:
\begin{enumerate}
\item $G=H$ if and only if $o(G-H)=\mathscr{P}$;
\item $G<H$ if and only if $o(G-H)=\mathscr{R}$;
\item $G>H$ if and only if $o(G-H)=\mathscr{L}$; and
\item $G\incomp H$ if and only if $o(G-H)=\mathscr{N}$.
\end{enumerate}
\end{fact}

We define $0$ to be the game $\{\emptyset\mid\emptyset\}$, so the game in which neither player has any available moves. Adding $0$ to any other games does not change it. Thus the above result in particular gives that $G\leq 0$ if and only if Right wins with Left moving first.

Many other game values beside $0$ are given names. The game $\{0\mid\emptyset\}$ is called $1$, while $\{\emptyset\mid 0\}$ is $-1$. Other integers are recursively defined as $n=\{n-1\mid\emptyset\}$ and $-n=\{\emptyset\mid -n+1\}$. A game which is a number is either an integer or a dyadic rational, for which the exact definition is not needed in this paper. The games $\{a\mid b\}$ with $a>b$ both numbers (we will only see integers) are called \textit{switches} and are often also written as $\displaystyle \frac{a+b}{2}\pm\frac{a-b}{2}$.

\begin{definition}
The \textit{Left stop} and \textit{Right stop} of a combinatorial game $G$, denoted by $LS(G)$ and $RS(G)$ respectively, are recursively defined as
\begin{align*}
LS(G)&=\begin{cases}
    	x & \text{if } G=x \text{ is a number},\\
    	\displaystyle \max_{G^L\in G^\mathcal{L}}\left\{RS\left(G^L\right)\right\} & \text{otherwise};
    \end{cases}\\
RS(G)&=\begin{cases}
    	x & \text{if } G=x \text{ is a number},\\
    	\displaystyle\min_{G^R\in G^\mathcal{R}}\left\{LS\left(G^R\right)\right\}& \text{otherwise}.
    \end{cases}
\end{align*}
\end{definition}

Note that $LS(G)\ge RS(G)$ for any game $G$. We will also use that $LS(-G)=-RS(G)$ and $RS(G)+LS(H)\leq LS(G+H)\leq LS(g)+LS(H)$. Further, for any number $x$, $G\leq x$ implies $LS(G)\leq x$; and $LS(G+x)=LS(G)+x$.

\begin{definition}
Two combinatorial games $G$ and $H$ are called \textit{infinitesimally close} if $LS(G-H)=0$ and $RS(G-H)=0$. A game infinitesimally close to 0 is called an \textit{infinitesimal}.
\end{definition}

The \textit{confusion interval} of $G$ is defined by $\mathcal{C}(G)=\{x\in\mathbb{D}: G\incomp x\}$. The endpoints of the confusion interval are the Left stop and Right stop. The measure of the confusion interval, $LS(G)-RS(G)$, is indicated by $\ell(G)$.

We can bound the measure of the confusion interval of a disjunctive sum of games based on those of the components:
\begin{lemma}\label{thm:ellDisjunctive}
For any two games $G$ and $H$ we have
\[\ell(G+H)\leq \ell(G)+\ell(H).\]
\end{lemma}
\begin{proof}
We have
\begin{align*}
\ell(G+H)&=LS(G+H)-RS(G+H)\\
&\leq LS(G)+LS(H)-RS(G)-RS(H)\\
&=\ell(G)+\ell(H).\qedhere
\end{align*}
\end{proof}

If all hot games were switches it would not be necessary to develop a temperature theory. In a disjunctive sum of switches, the best choice is simply one with the larger $\ell(G)$ (a kind of loves-me-loves-me-not situation). However, there are hot games such as $\{\{10\,\mid\,1\}\,\mid\,-1\}$,\textit{ containing threats}. The confusion intervals of $\{\{10\,\mid\,1\}\,\mid\,-1\}$ and $\{1\,\mid\,-1\}$ are the same, but, due to the threat, the urgency of going first in the former is larger than for the latter. Temperature theory exists because of such games (with threats); being an attempt to create a measurement system capable of estimating the ``quality of the threats'' present in the followers. Rulesets such as \snort or \domineering present switches, but few games with threats like $\{\{10\,\mid\,1\}\,\mid\,-1\}$. Still, even for these rulesets, it is not trivial to find upper bounds for the switches they contain.

We now define what it means to cool a combinatorial game, and what its temperature is.

\begin{definition}
Fix a combinatorial game $G$ and $t\ge -1$. Then \textit{$G$ cooled by $t$}, denoted $G_t$, is
\begin{itemize}
\item $G_t=n$, if $G$ is the integer $n$,
\item Set $\tilde{G}_t=\left\{G_t^\mathcal{L}-t\mid G_t^\mathcal{R}+t\right\}$, then $G_t=\tilde{G}_t$
if $G$ is not an integer and there is no $s<t$ such that $\tilde{G}_{s}$ is infinitesimally close to a number $x$;
\item $G_t=x$, if $G$ is not a number and there exists a $s<t$ such that $\tilde{G}_{s}$ is infinitesimally close to the number $x$ and $s$ is the smallest such.
\end{itemize}
\end{definition}
Note that in the last point, there is indeed a unique smallest $s$ such that $\tilde{G}_{s}$ is infinitesimally close to a number, but this is not immediate (see \cite[Section II.5]{Siegel2013} for more information).

\begin{definition}
The \textit{temperature} of a game $G$, denoted by $t(G)$ is the smallest $t\ge -1$ such that $G_t$ is infinitesimally close to a number.
\end{definition}

The temperature of a game intuitively measures the urgency of playing in it. A game with $t(G)<0$ is called \textit{cold}. This happens when a game is equal to a number, playing in it makes the situation worse, and thus the temperature is negative. A game with $t(G)=0$ is called \textit{tepid}. This happens when $LS(G)=RS(G)$ and the game $G$ is the number $LS(G)$ plus an infinitesimal. Note that there is no incentive to going first in a tepid or cold game.  Finally, a game with $t(G)>0$, a game in which there is an advantage to going first, is called \textit{hot}. 
Note that a game $G$ is hot if and only if $LS(G)>RS(G)$.

The number to which $G_{t(G)}$ is infinitesimally close to is the \textit{mean} of $G$ and indicated by $m(G)$.

The game $G_t$ for $t\leq t(G)$ can be thought of as playing $G$ but having to pay a penalty of $t$ ($-t$ for Left and $+t$ for Right) in the game and followers when making a move. Suppose that Left, going first, wins $G_t$. Left has to pay the penalty of $-t$ on her first move, and Right pays the penalty of $+t$ on his move. The two penalties cancel, and continue to do so, until Left makes her final, winning move. Thus in total Left loses $t$ moves and in the course of the game the number of moves won by Left must compensate for this loss. The temperature $t(G)$ is then the point after which the penalty is too high for either player to be interested in the game. Temperature thus gives a sense of how valuable in terms of gain a component in a disjunctive sum is to the players, or which component is the most urgent to move in.\footnote{Temperature can sometimes be misleading though. There are rare examples in which the component of highest temperature is not actually the most desirable one to move in.}
Thus a goal of computer scientists in combinatorial game theory is finding heuristics to evaluate the temperature \cite{MullerES2004,Cazenave2015}, as it points to a good, hopefully the best, move.

Given a disjunctive sum, we can bound the temperature based on the temperature of the components as in the next fact. This result will be used frequently in our calculations later on.
\begin{fact}[{\cite[Theorem II.5.18]{Siegel2013}}]
For all games $G$ and $H$ we have
\[t(G+H)\leq \max\{t(G),t(H)\}.\]
\end{fact}

When cooling a game, the Left stop and Right stop become closer, until they are equal to the mean. The behaviour of the stops can be seen visually by the graphical representation of the thermograph.

\begin{definition}
Given a game $G$, the ordered pair $(LS(G_t), RS(G_t))$, regarded as a pair of functions of $t$, is called the \textit{thermograph} of $G$.
\end{definition}
The graphical representation of a thermograph uses the following conventions: $LS(G_t)$ and $RS(G_t)$ are simultaneously plotted along the horizontal axis, with positive values on the left and negative on the right, while $t$ is plotted along the vertical axis. Note that $LS(G)$ and $RS(G)$ are the points at which the thermograph crosses the horizontal axis. The two sides meet at vertical value $t(G)$ and horizontal value $m(G)$, and are then topped by an infinite vertical mast.

By definition of the Left and Right stops we can construct the thermograph of a game inductively from the thermographs of the options. The left wall is the leftmost right wall of the thermographs of the options, sheared by subtracting $t$. The right wall is similarly the rightmost left wall of the options sheared by adding $t$. The mast begins where these two intersect.

\begin{example}
Consider the game $G=\{\{5\mid 2\}\mid\{-2\mid -3\}\}$. We will construct the thermograph of $G$ inductively from its options. The thermograph of a number $n$ is simply the vertical line at $n$. Thus the thermograph of $\{5\mid 2\}$ is as below, with the line at 5 sheared clockwise, the line at 2 counterclockwise, and the mast starting at their intersection.
\begin{center}
\includegraphics[scale=0.75]{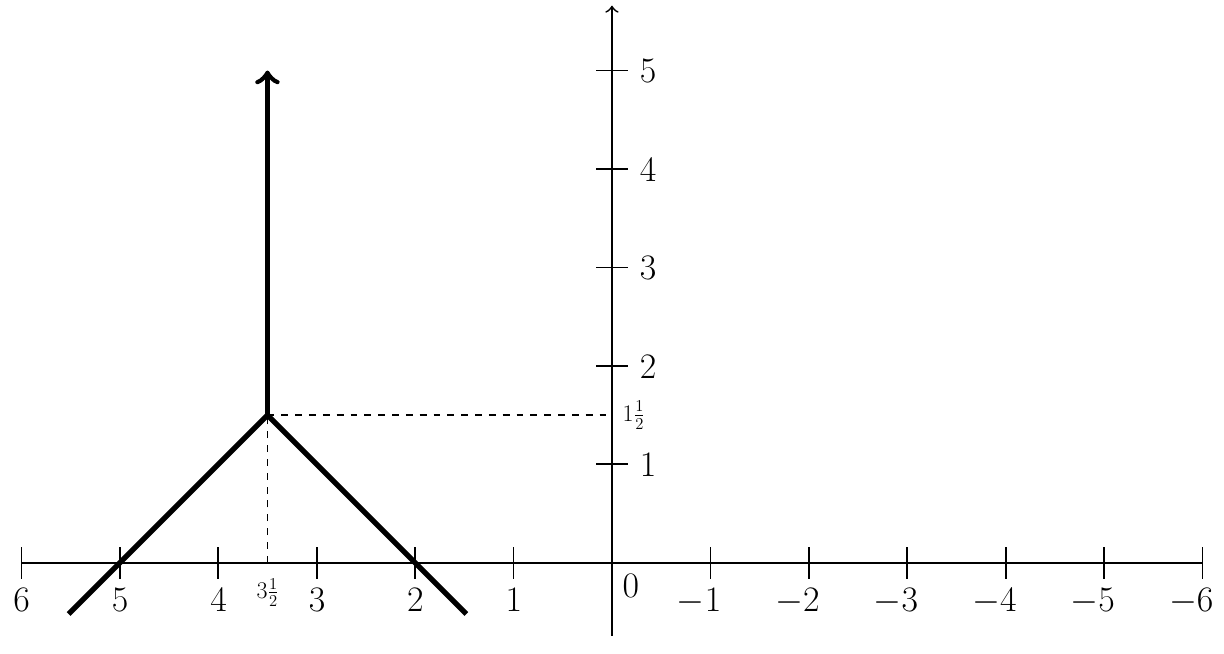}
\end{center}
The thermograph of $\{-2\mid -3\}$ is similarly given in the diagram below.
\begin{center}
\includegraphics[scale=0.75]{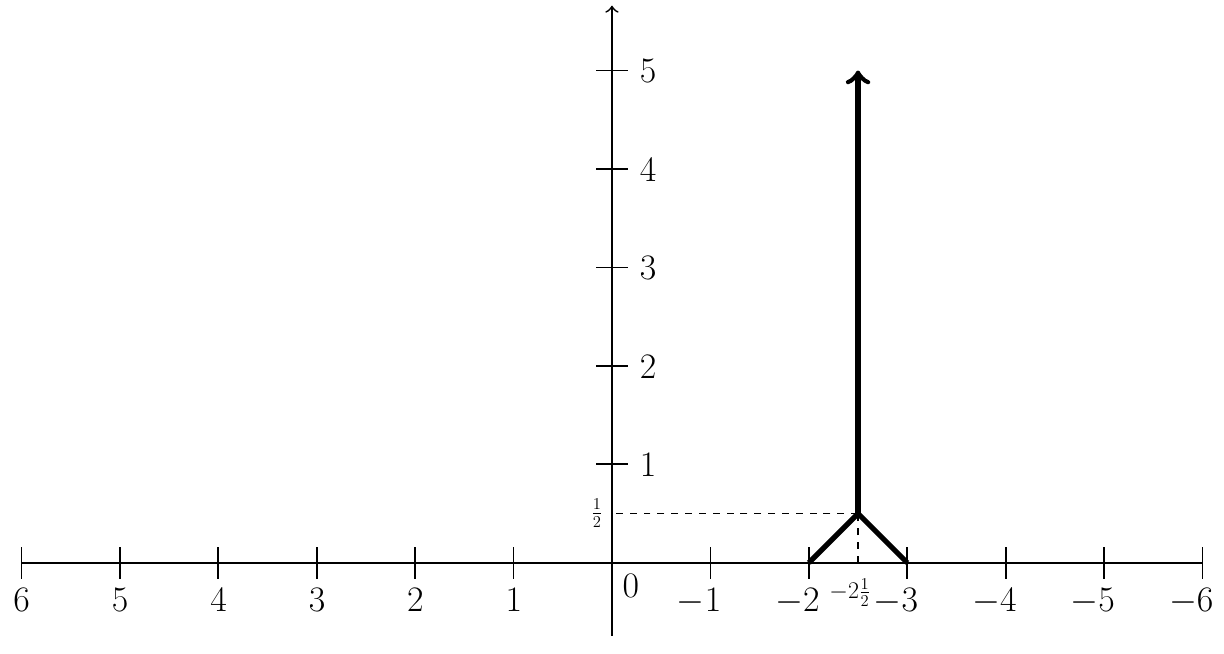}
\end{center}
To construct the thermograph of $G$, we take the right wall of the thermograph of $\{5\mid 2\}$ (as it gives the Right stop of the Left option) and shear it clockwise by subtracting $t$, and the left wall of the thermograph of $\{-2\mid -3\}$ and shear it counterclockwise by adding $t$, until their intersection, which is then topped by the mast. The thermograph of $G$ is then given below.
\begin{center}
\includegraphics[scale=0.75]{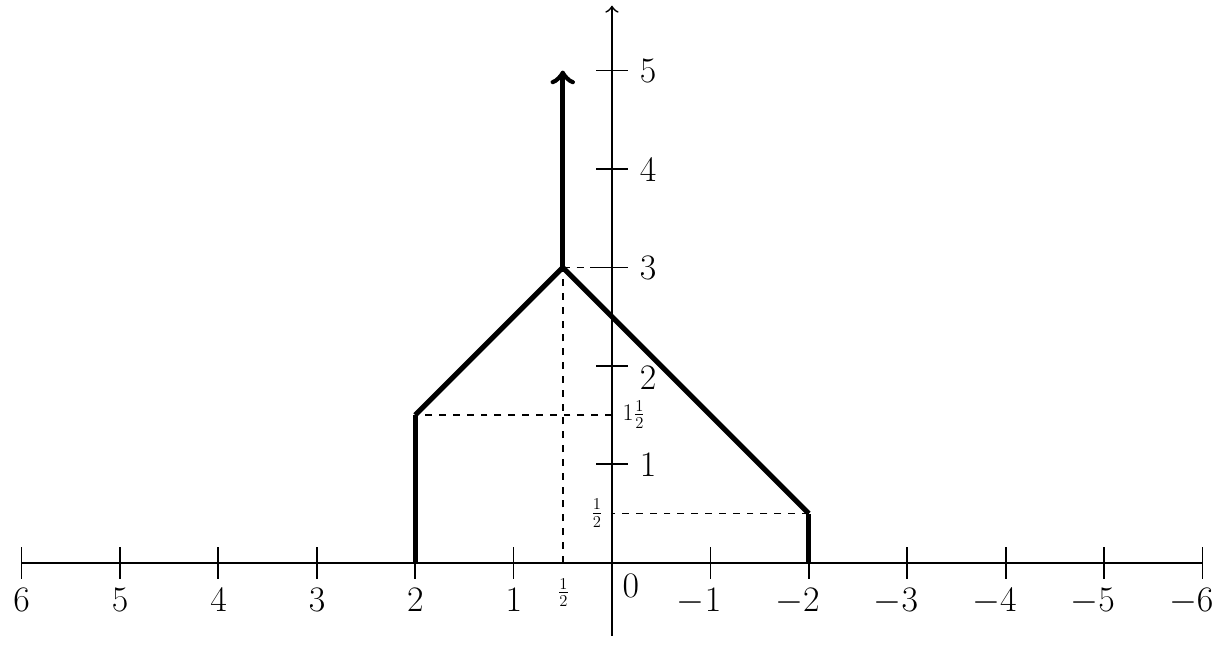}
\end{center}
\end{example}

Note that the temperature of a game $G$, being the vertical value at which the mast starts, is the length of all vertical segments plus the length of the oblique segments of either the right wall or the left wall above the horizontal axis and below the mast. We will use this property to bound the temperature of $G$. To do so, we will use $\ell(G)$.

Computer game playing programs often use temperature to find potentially good moves. As temperature is difficult to calculate though, it would be helpful to be able to bound the temperature of classes of games.

\begin{definition}
Given a class of games $S$, the \textit{boiling point} of $S$, denoted $BP(S)$, is the supremum of the temperatures of all games in $S$, thus
\[BP(S)=\sup_{G\in S}\left(t(G)\right).\]
\end{definition}
Of particular interest are the boiling points of classes of games which are the same ruleset played on different boards, often even the entire (infinite) set of boards. 

Historically, there has been much interest in temperatures of specific games. We will discuss previous work on the temperatures of \domineering and \snort in their respective sections. For information on temperatures of many other games, please see the recent survey by Berlekamp \cite{Berlekamp2019}. In particular, there has been no bound proven to hold for classes of games in general.

\section{An Upper Bound on the Boiling Point of a Game}\label{sec:BoilingPointBound}

In this section, we will demonstrate an upper bound on the boiling point of a class of games dependent solely on the maximum difference between Left and Right stops.

We will being by showing that for any game $G$ there exists a game $\widetilde{G}$ which has a single Left option and a single Right option such that $t(G)=t(\widetilde{G})$. 
\begin{theorem} 
Let $G$ be a hot game. Then, there are options $G^L$ and $G^R$ such that $t(G)=t(\{G^L\mid G^R\})$.
\end{theorem}

\begin{proof}
Consider $G_{t(G)}$, \ie $G$ cooled by its temperature. Since $G$ is hot we know that $G_{t(G)}$ is tepid, being equal to $m(G)$ plus some infinitesimal.

We have that $LS(G_{t(G)})=RS(G_{t(G)})=m(G)$ and those stops are achieved with Left and Right options of $G_{t(G)}$. These options are obtained from some $G^L$ and $G^R$, being $G^L_{t(G)}-t(G)$ and $G^R_{t(G)}+t(G)$. Therefore, $t(\{G^L\mid G^R\})=t(G)$.
\end{proof}

\begin{definition}\label{def:Thermic}
Let $G$ be a hot game. Then $(G^L,G^R)$ is called a pair of \textit{thermic options} of $G$ if $t(G)=t(\{G^L\mid G^R\})$. We set $\widetilde{G}=\{\widetilde{G}^L\mid \widetilde{G}^R\}$ where $\widetilde{G}^L=G^L$ and $\widetilde{G}^R=G^R$ are thermic options. We say {$\widetilde{G}$} is a \textit{thermic version} of $G$.
\end{definition}

When bounding temperatures, instead of working with $G$ we can consider a thermic version $\widetilde{G}$. This has the advantage that when we are constructing the thermograph, we only need to consider the thermographs of the unique options, rather than having to find which option has the leftmost right wall at every point.

In general, it is tedious to find a thermic version of a game $G$ as, in a worst-case scenario, all temperatures of combinations of Left and Right options would have to be checked. We will be using the thermic version for theoretical purposes only though, and for calculations will return to the original game.

Note that a thermic version of a game is not necessarily unique, as demonstrated in the following example:

\begin{example}
Consider $G=\{\{\{3\mid 1\}\mid 0\},\{2\mid 0\}\mid\{-1\mid -2\}\}$. Both $\widetilde{G}_1=\{\{\{3\mid 1\}\mid 0\}\mid\{-1\mid -2\}\}$ and $\widetilde{G}_2=\{\{2\mid 0\}\mid\{-1\mid -2\}\}$ are thermic versions of $G$.
\end{example}

Further, in many cases the thermographs of $G$ and a thermic version $\widetilde{G}$ are not identical, as shown in the following example. 

\begin{example}
Let $G=\{\{2\mid -1\},0\mid\{-2\mid -4\}\}$. The thermograph of $G$ is

\begin{center}
\includegraphics[scale=0.75]{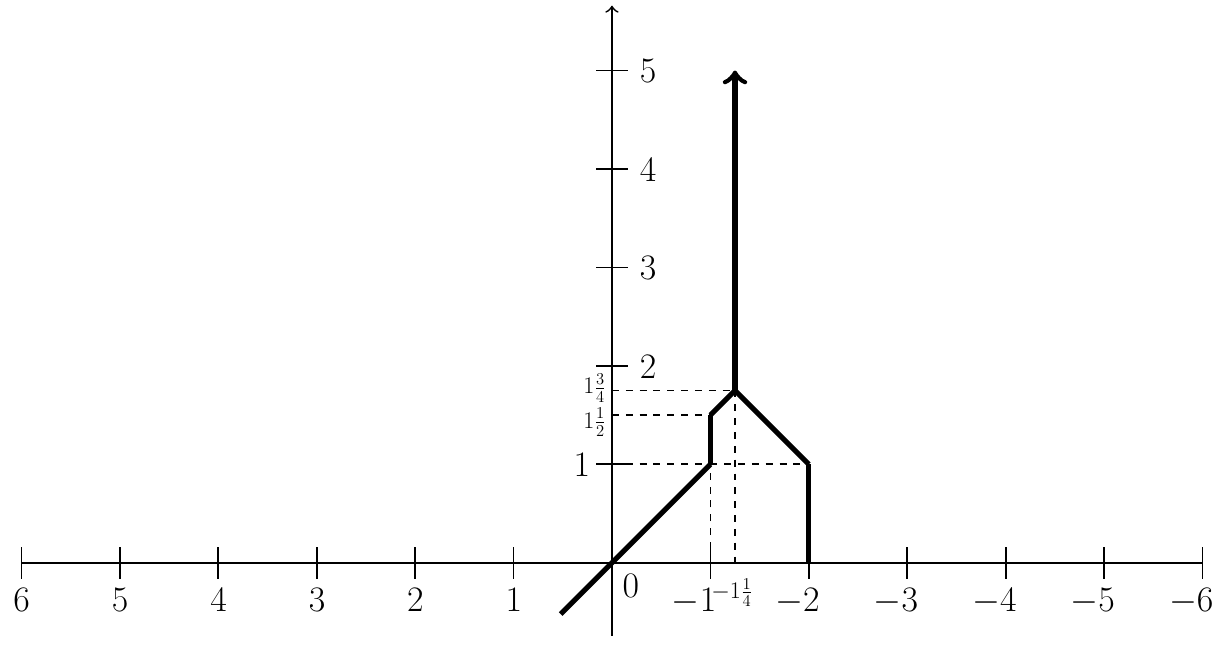}
\end{center}

A thermic version of $G$ is given by $\widetilde{G}=\{\{2\mid -1\}\mid\{-2\mid -4\}\}$. The thermograph of $\widetilde{G}$ is

\begin{center}
\includegraphics[scale=0.75]{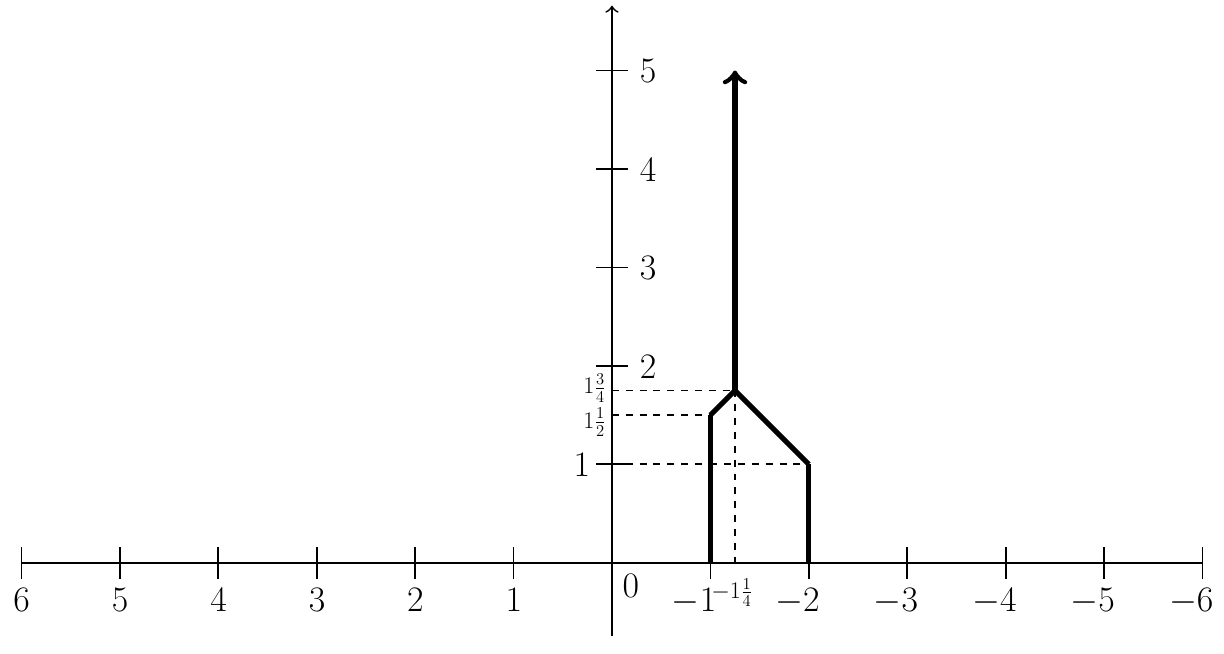}
\end{center}
\end{example}

Although the thermograph of a thermic version might not be identical to the one of the original game, the following proposition shows that it always lies under the thermograph of $G$:

\begin{proposition}
Let $G$ be a hot game and $\widetilde{G}$ a thermic version of $G$. For all $t\geq -1$ we have that $LS(\widetilde{G}_t)\leq LS(G_t)$ and $RS(\widetilde{G}_t)\geq RS(G_t)$. In particular $LS(\widetilde{G})\leq LS(G)$ and $RS(\widetilde{G})\geq RS(G)$.
\end{proposition}
\begin{proof}
Note first that $G_t$ and $\widetilde{G}_t$ are numbers only if $t\geq t(G)$, at which point the thermographs are identical.

For $t<t(G)$, we then have by definition $LS(G_t)=\max_{G^L}\left(RS(G^L_t)-t\right)$ and $LS(\widetilde{G}_t)=RS(\widetilde{G}^L_t)-t$ and similarly for the Right stops. Now since $\widetilde{G}^L$ is also a Left option of $G$, we have $LS(\widetilde{G}_t)\leq LS(G_t)$ and $RS(\widetilde{G}_t)\geq RS(G_t)$.

The special case of $t=0$ then follows immediately.
\end{proof}

\begin{corollary}
Given a hot game $G$ with a thermic version $\widetilde{G}$, we have
\[\ell(\widetilde{G})\leq \ell(G).\]
\end{corollary}
\begin{proof}
From the previous proposition we have
\[\ell(\widetilde{G})=LS(\widetilde{G})-RS(\widetilde{G})\leq LS(G)-RS(G)=\ell(G).\qedhere\]
\end{proof}

In this section, we will often consider segments of the thermograph. By length of such a segment we mean the change in $t$.

As we will be bounding the temperature of a game from the length of the vertical and oblique segments of the left and right walls of its thermograph, the turning points of the thermograph will be very important. We will be concentrating on the left wall throughout, but the same definitions and results apply to the right wall.
\begin{definition} 
Let $G$ be a hot game and $\widetilde{G}$ a thermic version of $G$. Let  $t_0=0,t_1,t_2,\ldots,t_k=t(G)$ be the sequence of the vertical coordinates of the turning points of the left boundary of the thermograph of $\widetilde{G}$. The Right stops of the sequence of $G^L(i)=\widetilde{G}^L_{t_i}-t_i$ define the segments of the boundary. If $RS(G^L(i+1))=RS(G^L(i))$, we have a vertical segment; on the other hand, if $RS(G^L(i+1))<RS(G^L(i))$, we have an oblique segment. Define
\begin{itemize}
  \item $t_i$ to be \textit{left vertical} if $RS(G^L(i+1))=RS(G^L(i))$;
  \item $t_i$ to be \textit{left oblique} if $RS(G^L(i+1))<RS(G^L(i))$.
\end{itemize}
We further define \[T^L_{\ver}=\sum_{\substack{t_i\text{ is left}\\\text{vertical}}} (t_{i+1}-t_i)\] and \[T^L_{\obl}=\sum_{\substack{t_i\text{ is left}\\\text{oblique}}} (t_{i+1}-t_i).\]
Right vertical, right oblique, $T_{\ver}^R$, and $T_{\obl}^R$ are defined similarly.
\end{definition}

Essentially, $T_{\ver}^L$ measures the length of the vertical segments of the left boundary between 0 and $t(G)$, while $T_{\obl}^L$ measures the oblique segments.

\begin{example}
Consider $G=\{\{\{6\mid 4\}\mid \{2\mid 0\}\}\mid\{\{0\mid -2\}\mid \{-4\mid -6\}\}\}$. Note that $G$ is its own thermic version as there are only a single Left option and a single Right option. The thermograph is given below:
\begin{center}
\includegraphics[scale=0.75]{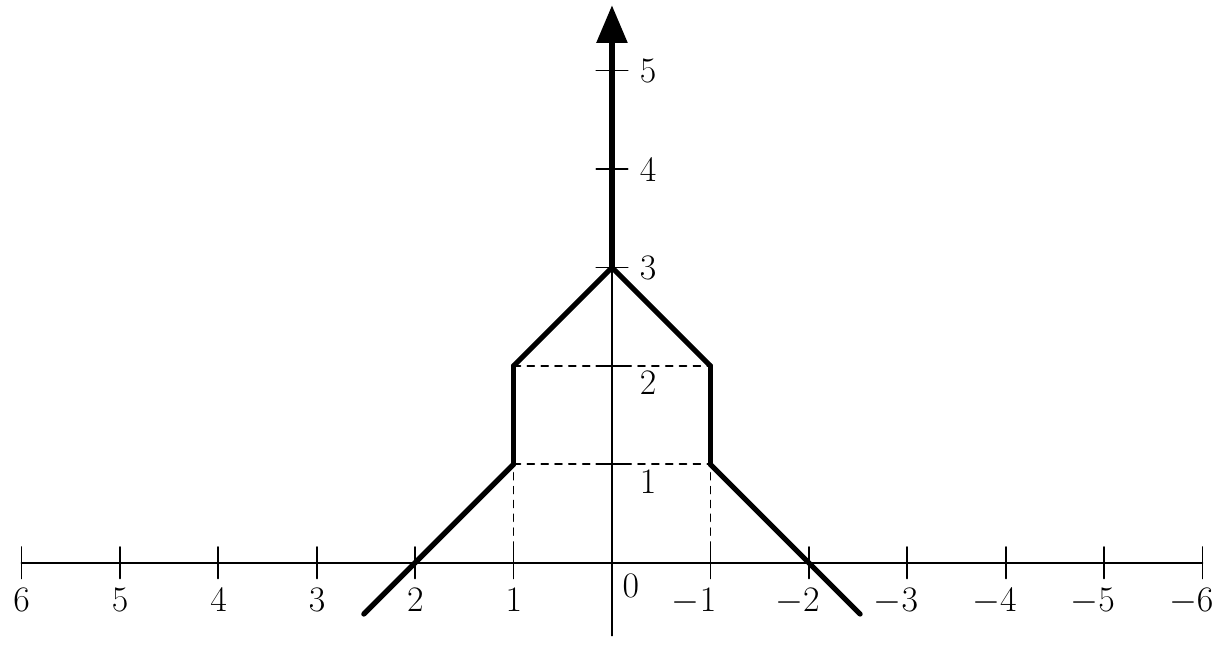}
\end{center}
The turning points of the left boundary are $t_0=0$, $t_1=1$, $t_2=2$, and $t_3=3$. We have
\begin{align*}
RS(G^L(0))&=RS(\{\{6\mid 4\}\mid\{2\mid 0\}\})=2\\
RS(G^L(1))&=RS(\{3*\mid 1*\})=1\\
RS(G^L(2))&=RS(\{1\mid 1\})=1\\
RS(G^L(3))&=RS(*)=0
\end{align*}

We have that $t_0=0$ and $t_2=2$ are left oblique and $t_1=1$ is left vertical. Thus
\[T^L_{\ver}=2-1=1 \quad\text{ and }\quad T^L_{\obl}=(1-0)+(3-2)=2.\]
Since the thermograph is symmetric, $T^R_{\ver}=1$ and $T^R_{\obl}=2$.
\end{example}

The following is an immediate consequence of the previous definition and that oblique segments have slope $\pm 1$.

\begin{lemma}
Given a hot game $G$ and a thermic version $\widetilde{G}$ we have
\[T^L_{\ver}+T^L_{\obl}=T^R_{\ver}+T^R_{\obl}=t(G)\]
and
\[\ell(\widetilde{G})=T^L_{\obl}+T^R_{\obl}.\]
\end{lemma}

We will now demonstrate the bound on the temperature of a game from the measure of its confusion interval and those of its options using the vertical and oblique segments of the thermograph.

\begin{theorem}\label{thm:TempBound}
Let $G$ be a hot game and $\widetilde{G}$ a thermic version of $G$. 
Then \[t(G)\leq \ell(H)+\frac{\ell(G)}{2}\]
where $H=\widetilde{G}^L$ if $T^L_{\ver}\geq T^R_{\ver}$ and $H=\widetilde{G}^R$ otherwise.
\end{theorem}

\begin{proof}
We will demonstrate this bound in the case of $T^L_{\ver}\geq T^R_{\ver}$. The second case follows similarly.

First consider $T_{\ver}^L$. Since the left boundary of the thermograph of $\widetilde{G}$ comes from the right boundary of the thermograph of $\widetilde{G}^L$, we know that $T_{\ver}^L$ is at most the length of the oblique segments of the latter. Further, we know that the length of these oblique segments are at most the distance between the Left and Right stops. Thus $T_{\ver}^L\leq \ell(\widetilde{G}^L)$.

On the other hand, $T^L_{\ver}+T^L_{\obl}=T^R_{\ver}+T^R_{\obl}$ and, by assumption, $T^L_{\ver}\geq T^R_{\ver}$. Hence, $T^L_{\obl}\leq T^R_{\obl}$. We have
\[2\times T^L_{\obl}\leq T^L_{\obl}+T^R_{\obl}=\ell(\widetilde{G})\leq \ell(G)\]
so that\[T^L_{\obl}=\frac{\ell(\widetilde{G})}{2}\leq \frac{\ell(G)}{2}.\]

Therefore,
$t(G)=T^L_{\ver}+T^L_{\obl}\leq \ell(\widetilde{G}^L)+\frac{\ell(G)}{2}$.
\end{proof}

Working with the thermic options above had the advantage that we only had to consider a single option for each player. Note that the thermic version of a game $G$ is difficult to determine in general though. We will instead bound the measure of the confusion interval for \textit{all} options in our applications in the next section. 

The following theorem is a direct result of the previous one. It is the first known theorem giving an upper bound on the boiling point of a class of games.
\begin{theorem}\label{thm:BPbound1}
Let $S$ be a class of short games and $J,K$ be two non-negative numbers. If for all $G\in S$, we have $\ell(G)\leq K$ and for all $G^L$ and $G^R$ that $\ell(G^L), \ell(G^R)\leq J$, then \[BP(S)\leq \frac{K}{2}+J.\]
\end{theorem}
Note that if $S$ is closed under options, then $K=J$.

The next example will demonstrate that this bound is tight in some cases.
\begin{example} 
Consider $S=\{G:\ell(G),\ell(G^L),\ell(G^R)\leq 6\}$ the set of short games $G$ for which $\ell(G)\leq 6$, closed under options.

By the previous theorem we know that $BP(S)\leq 9$. Consider the following sequence of games, all of which belong to $S$:
\begin{equation*}
\begin{aligned}
G_0&=\pm\{9\mid 3\} &t(G_0)=6\;\\
G_1&=\pm\{\{15\,|\,9\}\,|\,3\} &t(G_1)=\frac{15}{2}\\
G_2&=\pm\{\{\{21\,|\,15\}\,|\,9\}\,|\,3\} &t(G_2)=\frac{33}{4}\\
G_3&=\pm\{\{\{\{27\,|\,21\}\,|\,15\}\,|\,9\}\,|\,3\} &t(G_3)=\frac{69}{8}\\
&(\ldots)
\end{aligned}
\end{equation*}
For this sequence, we have $\displaystyle t(G_n)=9-\frac{3}{2^n}$, and as $n$ increases the temperature approaches 9. Therefore, $BP(S)=9$.
\end{example}

We will use \cref{thm:BPbound1} in the next section to give upper bounds on the boiling point for specific SP-games. The following proposition will be used to bound $\ell(G)$.
\begin{proposition}\label{thm:boundell}
Given a hot game $G$, if we know that $G^L-G-K+\epsilon\leq 0$ for all Left options $G^L$, a fixed number $K$, and some infinitesimal $\epsilon$, then $\ell(G)\leq K$.
\end{proposition}
\begin{proof}
Since $G^L-G+\epsilon\leq K$, we have $LS(G^L-G)\leq K$. Thus (remember that $LS(G)=RS(G^L)$ for some $G^L$):
\begin{align*}
\ell(G)&=LS(G)-RS(G)\\
&=RS(G^L)-RS(G)\\
&=RS(G^L)+LS(-G)\\
&\leq LS(G^L-G)\\
&\leq K\qedhere
\end{align*}
\end{proof}
 
The game $G^L-G-K+\epsilon$ corresponds, in some sense, to letting Left play twice in $G$, balancing that with $K-\epsilon$. If, for every game in that class, we can bound the effect of that second move, then we bound the confusion intervals for the class.  The infinitesimal is important since if Left moves the game or follower of $G^L-G$ to one of temperature 0, Right can respond in the infinitesimal instead of the integer. For example in the proof of \cref{thm:SnortTemp} is it possible to replace $K=5$ with $4+\uparrow$? If so, this would reduce the upper bound on the temperature to 6.

Our goal thus is to find the minimal number $K$ and an infinitesimal $\epsilon$ for which Right has a winning strategy going second in $G^L-G-K+\epsilon$.

As examples, we will show how to apply this to two games for specific classes of boards, \domineering snakes and \snort on a path, and discuss their temperatures on other boards. Note that since $t(G+H)\leq\max\{t(G),t(H)\}$, it is sufficient to bound the temperature when playing on a connected board.

\section{\domineering}\label{sec:domineering}
Berlekamp conjectured in the late 1980's that the boiling point of \domineering is 2 (see \cite{Guy1996,ShankarS2005}). The values and temperature of \domineering have been the main point of interest in many papers (see for example \cite{Berlekamp1988,Wolfe1993,Kim1995,Kim1996,ShankarS2005,DrummondCole2005,UiterwijkB2015}), all of which support Berlekamp's conjecture. It is known that if the conjecture is correct, then the bound is tight as the position below, found by Drummond-Cole in 2004 \cite{DrummondCole20xx}, has 
temperature 2.
\begin{center}
\begin{tikzpicture}[scale=0.75]
\draw[line width=1.25pt] (0,5)--(1,5)--(1,4)--(2,4)--(2,5)--(3,5)--(3,3)--(5,3)--(5,2)--(4,2)--(4,1)--(5,1)--(5,0)--(3,0)--(3,1)--(1,1)--(1,3)--(0,3)--cycle;
\draw (4,0)--(4,1)--(3,1)--(3,3)--(1,3)--(1,4)--(0,4);
\draw (2,1)--(2,4)--(3,4);
\draw (1,2)--(4,2)--(4,3);
\end{tikzpicture}
\end{center}
However, there has been no general theorem which states an upper bound for the temperature of \Domineering.

\subsection{\domineering Snakes}
We will start by considering \domineering snakes that fit within a $2\times n$ grid.

A \domineering snake is a \domineering position in which the board in some sense has `width' 1. They can be inductively constructed as follows: 
\begin{itemize}
\item[Step 1:] Place a single square.
\item[Step $n$:] Attach a new square at the top, right, or bottom edge of the square placed in step $n-1$. When doing so, no $2\times 2$ subgrid may be formed.
\end{itemize}

An example of a snake is below.
\begin{center}
\begin{tikzpicture}[line cap=round,line join=round,>=triangle 45,x=1.0cm,y=1.0cm,scale=0.5]
\draw (-3.,2.)-- (-3.,1);
\draw (-3.,1.)-- (0.,1.);
\draw (0.,1.)-- (0.,4.);
\draw (0.,4.)-- (2.,4.);
\draw (2.,4.)-- (2.,0.);
\draw (2.,0.)-- (5.,0.);
\draw (5.,0.)-- (5.,2.);
\draw (5.,2.)-- (4.,2.);
\draw (4.,2.)-- (4.,1.);
\draw (4.,1.)-- (3.,1.);
\draw (3.,1.)-- (3.,5.);
\draw (3.,5.)-- (-1.,5.);
\draw (-1.,5.)-- (-1.,2.);
\draw (-1.,2.)-- (-3.,2.);
\draw (-2.,2.)-- (-2.,1.);
\draw (-1.,2.)-- (-1.,1.);
\draw (-1.,2.)-- (0.,2.);
\draw (-1.,3.)-- (0.,3.);
\draw (-1.,4.)-- (0.,4.);
\draw (0.,4.)-- (0.,5.);
\draw (1.,5.)-- (1.,4.);
\draw (2.,4.)-- (2.,5.);
\draw (2.,5.)-- (2.,4.);
\draw (2.,4.)-- (3.,4.);
\draw (3.,3.)-- (2.,3.);
\draw (2.,2.)-- (3.,2.);
\draw (2.,1.)-- (3.,1.);
\draw (3.,1.)-- (3.,0.);
\draw (4.,1.)-- (4.,0.);
\draw (4.,1.)-- (5.,1.);
\end{tikzpicture}
\end{center}

\domineering snakes are interesting as any move, whether by Left or by Right, results in a disjunctive sum of two smaller snakes. Thus they are amenable to a recursive study. Further, they often naturally occur during play on larger grids.

Conway in 1976 in the first edition of \cite[pp.114-121]{Conway2001} gives a characterization of snakes in which a square is added to the right alternating with a square up. His results show that such snakes have temperature at most 1.

Wolfe in 1993 \cite{Wolfe1993} gives reductions that show it is sufficient to consider snakes in which at most 4 squares are added in either direction. He proceeds to give values of all periodic snakes - those snakes in which, after an initial chain, the number of squares added vertically is always the same, as is the number of squares added horizontally. And finally, values are given for many repeating snakes fitting within a $3\times n$ grid, so snakes in which always two squares are added horizontally.

The snakes we consider fit into a $2\times n$ grid, thus they are snakes where always at least two squares are added to the right and exactly one square horizontally, alternating between up and down. We do not make any assumptions on repeating patterns though, so that many of the cases considered here are not covered by the results by Wolfe in 1993 \cite{Wolfe1993}.

As an example, the snake below on the left is considered as fitting into a $2\times n$ grid as it can be folded by alternating vertical addition up and down into the snake on the right without changing the game.

\begin{center}
\begin{tikzpicture}[line cap=round,line join=round,>=triangle 45,x=1.0cm,y=1.0cm,scale=0.75]
\draw (-2.,-2.) -- (-1.,-2.) -- (-1.,-1.) -- (1.,-1.) -- (1.,0.) -- (4.,0.) -- (4.,2.) -- (3.,2.) -- (3.,1.) -- (0.,1.) -- (0.,0.) -- (-2.,0.) -- cycle;
\draw (7.,-2.) -- (7.,0.) -- (10.,0.) -- (10.,-1.) -- (12.,-1.) -- (12.,0.) -- (13.,0.) -- (13.,-2.) -- (9.,-2.) -- (9.,-1.) -- (8.,-1.) -- (8.,-2.) -- cycle;
\draw  (-2.,-2.)-- (-1.,-2.);
\draw  (-1.,-2.)-- (-1.,-1.);
\draw  (-1.,-1.)-- (1.,-1.);
\draw  (1.,-1.)-- (1.,0.);
\draw  (1.,0.)-- (4.,0.);
\draw  (4.,0.)-- (4.,2.);
\draw  (4.,2.)-- (3.,2.);
\draw  (3.,2.)-- (3.,1.);
\draw  (3.,1.)-- (0.,1.);
\draw  (0.,1.)-- (0.,0.);
\draw  (0.,0.)-- (-2.,0.);
\draw  (-2.,0.)-- (-2.,-2.);
\draw  (7.,-2.)-- (7.,0.);
\draw  (7.,0.)-- (10.,0.);
\draw  (10.,0.)-- (10.,-1.);
\draw  (10.,-1.)-- (12.,-1.);
\draw  (12.,-1.)-- (12.,0.);
\draw  (12.,0.)-- (13.,0.);
\draw  (13.,0.)-- (13.,-2.);
\draw  (13.,-2.)-- (9.,-2.);
\draw  (9.,-2.)-- (9.,-1.);
\draw  (9.,-1.)-- (8.,-1.);
\draw  (8.,-1.)-- (8.,-2.);
\draw  (8.,-2.)-- (7.,-2.);
\draw  (-2.,-1.)-- (-1.,-1.);
\draw  (-1.,-1.)-- (-1.,0.);
\draw  (0.,0.)-- (0.,-1.);
\draw  (0.,0.)-- (1.,0.);
\draw  (1.,0.)-- (1.,1.);
\draw  (2.,1.)-- (2.,0.);
\draw  (3.,1.)-- (3.,0.);
\draw  (3.,1.)-- (4.,1.);
\draw  (7.,-1.)-- (8.,-1.);
\draw  (8.,-1.)-- (8.,0.);
\draw  (9.,0.)-- (9.,-1.);
\draw  (9.,-1.)-- (10.,-1.);
\draw  (10.,-1.)-- (10.,-2.);
\draw  (11.,-1.)-- (11.,-2.);
\draw  (12.,-1.)-- (12.,-2.);
\draw  (12.,-1.)-- (13.,-1.);
\draw (5.36,-0.1) node[anchor=north west] {$=$};
\end{tikzpicture}
\end{center}

On the other hand, the snake below does not fit into a $2\times n$ grid since folding it the same way results in a $2\times 2$ subgrid, thus changing the game.

\begin{center}
\begin{tikzpicture}[line cap=round,line join=round,>=triangle 45,x=1.0cm,y=1.0cm, scale=0.75]
\draw (0.,-7.) -- (0.,-5.) -- (1.,-5.) -- (1.,-4.) -- (4.,-4.) -- (4.,-3.) -- (5.,-3.) -- (5.,-5.) -- (2.,-5.) -- (2.,-6.) -- (1.,-6.) -- (1.,-7.) -- cycle;
\draw (8.,-7.) -- (8.,-5.) -- (10.,-5.) -- (10.,-6.) -- (12.,-6.) -- (12.,-5.) -- (13.,-5.) -- (13.,-7.) -- cycle;
\draw (6.3,-5.04) node[anchor=north west] {$=$};
\draw  (0.,-7.)-- (0.,-5.);
\draw  (0.,-5.)-- (1.,-5.);
\draw  (1.,-5.)-- (1.,-4.);
\draw  (1.,-4.)-- (4.,-4.);
\draw  (4.,-4.)-- (4.,-3.);
\draw  (4.,-3.)-- (5.,-3.);
\draw  (5.,-3.)-- (5.,-5.);
\draw  (5.,-5.)-- (2.,-5.);
\draw  (2.,-5.)-- (2.,-6.);
\draw  (2.,-6.)-- (1.,-6.);
\draw  (1.,-6.)-- (1.,-7.);
\draw  (1.,-7.)-- (0.,-7.);
\draw  (8.,-7.)-- (8.,-5.);
\draw  (8.,-5.)-- (10.,-5.);
\draw  (10.,-5.)-- (10.,-6.);
\draw  (10.,-6.)-- (12.,-6.);
\draw  (12.,-6.)-- (12.,-5.);
\draw  (12.,-5.)-- (13.,-5.);
\draw  (13.,-5.)-- (13.,-7.);
\draw  (13.,-7.)-- (8.,-7.);
\draw  (1.,-6.)-- (0.,-6.);
\draw  (1.,-5.)-- (1.,-6.);
\draw  (1.,-5.)-- (2.,-5.);
\draw  (2.,-5.)-- (2.,-4.);
\draw  (3.,-4.)-- (3.,-5.);
\draw  (4.,-4.)-- (4.,-5.);
\draw  (4.,-4.)-- (5.,-4.);
\draw  (9.,-5.)-- (9.,-7.);
\draw  (8.,-6.)-- (10.,-6.);
\draw  (10.,-6.)-- (10.,-7.);
\draw  (11.,-6.)-- (11.,-7.);
\draw  (12.,-6.)-- (12.,-7.);
\draw  (12.,-6.)-- (13.,-6.);
\end{tikzpicture}
\end{center}

For our snakes any move results in a disjunctive sum of two smaller games, which are still \domineering snakes fitting within a $2\times n$ grid. We will take advantage of this in the following proof.

\begin{proposition}\label{thm:BODomSnakes}
The boiling point of \domineering played on a snake fitting within a $2\times n$ grid is at most 3.
\end{proposition}
\begin{proof}
Let $G$ be \domineering played on a snake fitting within a $2\times n$ grid. We will show that $H^L-H-2$ is a Right win if Left goes first for $H=G,G^L,G^R$. By \cref{thm:boundell} we then have $\ell(H)\leq 2$. The result follows by \cref{thm:BPbound1}.

Note that when $H=G^L$ or $H=G^R$ the move already made results in each component breaking into a disjunctive sum at the same spot, two of which are the negative of each other, the other two have the difference of the 0th Left move. In the case of $H=G^R$, we then have
\[-G^R=-G_1-G_2\qquad\text{and}\qquad G^{RL}=G_1^L+G_2.\] 
This gives \[H^L-H=G^{RL}-G^R=G_1^L+G_2-G_1-G_2=G_1^L-G_1,\] so that we reduce the game to being played on a smaller board. The case $H=G^L$ is similarly a disjunctive sum which reduces to a smaller board.

Thus it is sufficient if we give a winning strategy for Right when $H=G$, which we will do next.

We will label the columns of the $2\times n$ grid containing the snake as $1,\ldots, n$ and use the convention that the move already made in $G^L$ is in column $k$ and called move 0. The rows in $G^L$ are labelled $1$ and $2$ and those in $-G$ are labelled $a$ and $b$. We will denote a move occupying the two squares $(x,y)$ and $(u,v)$ by $\{(x,y),(u,v)\}$. For example move 0 would be $\{(k,1),(k,2)\}$. Without loss of generality we will also assume that the square adjacent to move 0 on the left is in row 2 as below.

\begin{center}
\begin{tikzpicture}[scale=0.75]
\foreach \y in {0,-3}{
\draw (0,\y+1)--(11,\y+1);
\draw (0,\y)--(2,\y); \draw (4,\y)--(7,\y); \draw (9,\y)--(11,\y);
\draw (1,\y+2)--(5,\y+2); \draw (6,\y+2)--(10,\y+2);
\foreach \x in {0,1,2,4,5,6,7,9,10,11}{
	\draw (\x,\y)--(\x,\y+1);}
\foreach \x in {1,2,3,4,5,6,7,8,9,10}{
	\draw (\x,\y+1)--(\x,\y+2);}}
\Ldomino{6}{0}
\node at (-1,-2) {$-G$};
\node at (-1,1) {$G^L$};
\begin{scriptsize}
\node at (0.5,2.5) {$1$};
\node at (1.5,2.5) {$2$};
\node at (6.5,2.5) {$k$};
\node at (10.5,2.5) {$n$};
\node at (11.5,1.5) {$1$};
\node at (11.5,0.5) {$2$};
\node at (11.5,-1.5) {$a$};
\node at (11.5,-2.5) {$b$};
\end{scriptsize}
\end{tikzpicture}
\end{center}

If Left plays her first move in $G^L$ or in $-G$ without overlapping column $k$, then Right can mimic the move in the other game, resulting in a disjunctive sum and reducing the board to a smaller size. We may thus assume that Left's move is $(a\, k, a\, k+1)$ or in $(b\, k-1, b\, k)$. When Left makes this move, Right will respond by moving in $-2$, and then in $-1$ when Left makes the second of these moves. Note that due to the form of snakes we have chosen, Left cannot overlap either of these moves again, so that Right can continue a mimicking strategy until the end of the game and wins.
\end{proof}

Note that for \domineering snakes in general any move results in a disjunctive sum, but for the above strategy to work we do require the snake fits into a $2\times n$ grid. If this were not the case, then Left could potentially overlap a move overlapping move 0 again.

\subsection{\domineering on other Boards}

Berlekamp in 1988 \cite{Berlekamp1988} studied \domineering on $2\times n$ and $3\times n$ grids. For $2\times n$, he gives all values, correct up to infinitesimals, in terms of heating and overheating operators, which in some sense are inverses of cooling. In particular he shows that the values are periodic for sufficiently large $n$. From this, we get the exact temperatures for \domineering $2\times n$ as follows:

\begin{center}
\begin{tabular}{l| *5l}
$n$ & 1 & 2 & 3 & 4 & 5\\\hline
$t$ & 0 & 1 & 5/4 & 0 & 0
\end{tabular}
\end{center}
And for $n=10k+j$ with $k\in\mathbb{Z}_{\geq 0}$
\begin{center}
\begin{tabular}{l| *9l l}
$j$ & $6$ & $7$ & $8$ & $9$ & $10$ & $11$ & $12$ & $13$ & $14$ & $15$\\\hline
$t$ & 1 & 1 & 9/8 & 9/8 & 19/16 & 19/16 & 0 & 0 & 9/8 & 9/8 
\end{tabular}
\end{center}

\section{\Snort}\label{sec:Snort}
For \Snort, values of some positions are known, and thus their temperatures. If the board is the star $K_{1,n}$, then the value is $\{n\mid -n\}$ and thus the temperature is $n$. So, in general, the boiling point of \snort is infinite. However, there are no general results if the maximum degree of the graph is bounded. We will give an upper bound on the temperature when playing on a path and a conjecture on the boiling point based on the degree of the board.

\subsection{\snort on a Path}
Any position of \snort on a path is a disjunctive sum of paths with the only potential pieces at the ends. See for example the position below.

\begin{center}
\begin{tikzpicture}[scale=0.75]
\foreach \x in {0,1,2,3,4,5,6,7,8}{
	\draw (\x,0)--(\x,1);}
\foreach \y in {0,1}{
	\draw (0,\y)--(8,\y);}
\Ltoken{2}{0}
\Rtoken{6}{0}
\begin{scope}[shift={(0,-2)}]
\foreach \x in {0,1,2,3,4,5,6,7,8,9,10,11,12}{
	\draw (\x,0)--(\x,1);}
\foreach \y in {0,1}{
	\draw (0,\y)--(3,\y);
	\draw (4,\y)--(9,\y);
	\draw (10,\y)--(12,\y);}
\Ltoken{2}{0}
\Ltoken{4}{0}
\Rtoken{8}{0}
\Rtoken{10}{0}
\node at (-1,0.5) {$=$};
\node at (3.5,0.5) {$+$};
\node at (9.5,0.5) {$+$};
\end{scope}
\end{tikzpicture}
\end{center}

So to determine the boiling point of \snort played on a path it is sufficient to look at such positions.

Using CGSuite, we have the following temperatures for \snort on a path of $n$ vertices, potentially with pieces placed on one or both ends. Here $LP_kR$ indicates a path of length $k+2$ with a Left piece in the left-most space and a Right piece in the right-most space, and similarly for the other notations. Empty values are due to positions that are not possible.
\begin{center}
\begin{tabular}{r|cccccccccccc}
$n$&1&2&3&4&5&6&7&8&9&10&11&12\\\hline
$P_n$&0&1&2&3/2&1&0&1&2&2&3/2&3/2&1\\\hline
$LP_{n-1}$&$-1$&$-1$&1/2&3/2&2 & 7/4&3/2&1&15/8&2&2&31/16\\\hline
$LP_{n-2}L$& &$-1$&$-1$&$-1$&1&3/2&2&3/2 & 7/4 & 1 & 7/4 & 15/8 \\\hline
$LP_{n-2}R$& & &$-1$&0&1&2&2&2 &1& 1 & 1 & 2 
\end{tabular}
\end{center}

Winning Ways \cite{BerlekampCG2004} contains a list of values for various \snort positions. This includes values of $LP_nL$ and $LP_nR$ cooled by 1, which have temperatures at most 1, implying that the original games have temperature at most 2. In general, we conjecture that the boiling point for \snort on a path is 2, and prove below that it is at most 7.5.

\begin{theorem}\label{thm:SnortTemp}
The temperature of \snort played on a path is at most $\displaystyle 7\frac{1}{2}$.
\end{theorem}

\begin{proof}
The temperatures for paths of length up to $14$ are above, showing in particular that the theorem holds up to length 4. For paths of length at least 5 we will show that Right wins $H^L-H-5$ with Left going first, giving $\ell(H)\leq 5$.

We label the vertices in $H^L$ as $1,\ldots, n$, with move 0 in vertex $k$, and the vertices in $-H$ as $1',\ldots, n'$.
\begin{center}
\begin{tikzpicture}[scale=0.75]
	\foreach \x in {-5,-4,-3,-1,0,1,2,3,5,6,7}{
	\draw (\x,0)--(\x,1);
	\draw (\x,2)--(\x,3);}
\foreach \y in {0,1,2,3}{
	\draw (-5,\y)--(7,\y);}
\draw[dash pattern=on 2pt off 2pt] (-2.5,0.5)--(-1.5,0.5);
\draw[dash pattern=on 2pt off 2pt] (3.5,0.5)--(4.5,0.5);
\draw[dash pattern=on 2pt off 2pt] (-2.5,2.5)--(-1.5,2.5);
\draw[dash pattern=on 2pt off 2pt] (3.5,2.5)--(4.5,2.5);
\Ltoken{0}{2}
\node at (-6,2.5) {$H^L$};
\node at (-6, 0.5) {$-H$};
\begin{scriptsize}
\node at (-4.5,3.5) {$1$};
\node at (-3.5,3.5) {$2$};
\node at (-0.5,3.5) {$k-1$};
\node at (0.5,3.5) {$k$};
\node at (1.5,3.5) {$k+1$};
\node at (2.5,3.5) {$k+2$};
\node at (5.5,3.5) {$n-1$};
\node at (6.5,3.5) {$n$};
\node at (-4.5,-0.5) {$1'$};
\node at (-3.5,-0.5) {$2'$};
\node at (-0.5,-0.5) {$k-1'$};
\node at (0.5,-0.5) {$k'$};
\node at (1.5,-0.5) {$k+1'$};
\node at (2.5,-0.5) {$k+2'$};
\node at (5.5,-0.5) {$n-1'$};
\node at (6.5,-0.5) {$n'$};
\end{scriptsize}
\end{tikzpicture}
\end{center}

Whenever possible, Right copies Left's move in the other component, resulting in a disjunctive sum. By induction we can thus assume the boards are empty except for move 0 and possibly at the ends of the path (which would be identical), until Left makes a move which Right cannot copy. The only moves by Left that cannot be copied by Right are $k-1'$, $k'$, and $k+1'$. Our strategy for Right is to essentially bound the area around move 0 in which Right cannot copy Left's moves, which we call the influence area.

\textit{Case 1: Left moves in $k'$}

(Note, if at any point $k-1$ or $k+1$ had already been played, Left cannot make this move.)

Right responds by playing $k+2$ or $k+2'$ if possible. If Left responds with the other move, Right then moves in $k-2$ or $k-2'$. If Left copies this move again, Right responds by playing in the integer.

\begin{center}
\begin{tikzpicture}[scale=0.75]
	\foreach \x in {-1,0,1,2,3}{
	\draw (\x,0)--(\x,1);
	\draw (\x,2)--(\x,3);}
\foreach \y in {0,1,2,3}{
	\draw (-2,\y)--(4,\y);}
\draw[dash pattern=on 2pt off 2pt] (-2.5,0.5)--(-1.5,0.5);
\draw[dash pattern=on 2pt off 2pt] (3.5,0.5)--(4.5,0.5);
\draw[dash pattern=on 2pt off 2pt] (-2.5,2.5)--(-1.5,2.5);
\draw[dash pattern=on 2pt off 2pt] (3.5,2.5)--(4.5,2.5);
\Ltoken{0}{2}
\Rtoken{0}{0}
\Rtoken{2}{0}
\Rtoken{2}{2}
\begin{scriptsize}
\node at (0.5,3.5) {$k$};
\node at (0.5,-0.5) {$k'$};
\end{scriptsize}
\node at (5,1.5) {or};

\begin{scope}[shift={(8,0)}]
	\foreach \x in {-1,0,1,2,3}{
	\draw (\x,0)--(\x,1);
	\draw (\x,2)--(\x,3);}
\foreach \y in {0,1,2,3}{
	\draw (-2,\y)--(4,\y);}
\draw[dash pattern=on 2pt off 2pt] (-2.5,0.5)--(-1.5,0.5);
\draw[dash pattern=on 2pt off 2pt] (3.5,0.5)--(4.5,0.5);
\draw[dash pattern=on 2pt off 2pt] (-2.5,2.5)--(-1.5,2.5);
\draw[dash pattern=on 2pt off 2pt] (3.5,2.5)--(4.5,2.5);
\Ltoken{0}{2}
\Rtoken{0}{0}
\Ltoken{2}{0}
\Ltoken{2}{2}
\begin{scriptsize}
\node at (0.5,3.5) {$k$};
\node at (0.5,-0.5) {$k'$};
\end{scriptsize}
\end{scope}
\end{tikzpicture}

Response to the right of move 0 (symmetric to the left)
\end{center}

The only time these responses are not possible for Right is if move 0 is too close to either end, which is only possible on one end since we assume $n\geq 5$. The end then looks like either of the two cases below.

\begin{center}
\begin{tikzpicture}[scale=0.75]
	\foreach \x in {-1,0,1,2}{
	\draw (\x,0)--(\x,1);
	\draw (\x,2)--(\x,3);}
\foreach \y in {0,1,2,3}{
	\draw (-2,\y)--(2,\y);}
\draw[dash pattern=on 2pt off 2pt] (-2.5,0.5)--(-1.5,0.5);
\draw[dash pattern=on 2pt off 2pt] (-2.5,2.5)--(-1.5,2.5);
\Ltoken{0}{2}
\Rtoken{0}{0}
\begin{scriptsize}
\node at (0.5,3.5) {$k$};
\node at (0.5,-0.5) {$k'$};
\end{scriptsize}
\node at (3,1.5) {or};

\begin{scope}[shift={(5,0)}]
	\foreach \x in {-1,0,1}{
	\draw (\x,0)--(\x,1);
	\draw (\x,2)--(\x,3);}
\foreach \y in {0,1,2,3}{
	\draw (-2,\y)--(1,\y);}
\draw[dash pattern=on 2pt off 2pt] (-2.5,0.5)--(-1.5,0.5);
\draw[dash pattern=on 2pt off 2pt] (-2.5,2.5)--(-1.5,2.5);
\Ltoken{0}{2}
\Rtoken{0}{0}
\begin{scriptsize}
\node at (0.5,3.5) {$k$};
\node at (0.5,-0.5) {$k'$};
\end{scriptsize}
\end{scope}
\end{tikzpicture}
\end{center}

Left then has at most three additional non-copiable moves (since $n\geq 5$ move 0 can only be 2 away from at most one end) to which Right responds in the integer. Since Right has to respond in the integer at most four times, he wins.

\textit{Case 2: Left moves in $k+1'$ (wlog)}

If $k=1$ or $R$ has previously been played in $k-2$ and $k-2'$, the influence area is already bounded to the left of move 0. Otherwise Right moves in $k-1'$.

\begin{center}
\begin{tikzpicture}[scale=0.75]
	\foreach \x in {-1,0,1,2,3}{
	\draw (\x,0)--(\x,1);
	\draw (\x,2)--(\x,3);}
\foreach \y in {0,1,2,3}{
	\draw (-2,\y)--(4,\y);}
\draw[dash pattern=on 2pt off 2pt] (-2.5,0.5)--(-1.5,0.5);
\draw[dash pattern=on 2pt off 2pt] (3.5,0.5)--(4.5,0.5);
\draw[dash pattern=on 2pt off 2pt] (-2.5,2.5)--(-1.5,2.5);
\draw[dash pattern=on 2pt off 2pt] (3.5,2.5)--(4.5,2.5);
\Ltoken{0}{2}
\Ltoken{0}{0}
\Rtoken{2}{0}
\Ltoken{1}{2}
\begin{scriptsize}
\node at (1.5,3.5) {$k$};
\node at (2.5,-0.5) {$k+1'$};
\end{scriptsize}
\node at (5,1.5) {or};

\begin{scope}[shift={(8,0)}]
	\foreach \x in {-1,0,1,2,3}{
	\draw (\x,0)--(\x,1);
	\draw (\x,2)--(\x,3);}
\foreach \y in {0,1,2,3}{
	\draw (-2,\y)--(4,\y);}
\draw[dash pattern=on 2pt off 2pt] (-2.5,0.5)--(-1.5,0.5);
\draw[dash pattern=on 2pt off 2pt] (3.5,0.5)--(4.5,0.5);
\draw[dash pattern=on 2pt off 2pt] (-2.5,2.5)--(-1.5,2.5);
\draw[dash pattern=on 2pt off 2pt] (3.5,2.5)--(4.5,2.5);
\Ltoken{1}{2}
\Rtoken{-1}{0}
\Rtoken{2}{0}
\Rtoken{-1}{2}
\begin{scriptsize}
\node at (1.5,3.5) {$k$};
\node at (2.5,-0.5) {$k+1'$};
\end{scriptsize}
\end{scope}
\end{tikzpicture}

To the left of move 0
\end{center}

When Left moves in $k-1$ or the influence area was already bounded to the left, Right responds by playing $k+2$, unless $k=n-2$ or $L$ has been played in $k+3$ and $k+3'$. When Left copies and moves in $k-1'$, or the influence area was bounded to the right, Right moves in the integer.

\begin{center}
\begin{tikzpicture}[scale=0.75]
	\foreach \x in {-1,0,1,2,3}{
	\draw (\x,0)--(\x,1);
	\draw (\x,2)--(\x,3);}
\foreach \y in {0,1,2,3}{
	\draw (-2,\y)--(4,\y);}
\draw[dash pattern=on 2pt off 2pt] (-2.5,0.5)--(-1.5,0.5);
\draw[dash pattern=on 2pt off 2pt] (3.5,0.5)--(4.5,0.5);
\draw[dash pattern=on 2pt off 2pt] (-2.5,2.5)--(-1.5,2.5);
\draw[dash pattern=on 2pt off 2pt] (3.5,2.5)--(4.5,2.5);
\Ltoken{-1}{2}
\Rtoken{0}{0}
\Rtoken{1}{0}
\Rtoken{1}{2}
\begin{scriptsize}
\node at (-0.5,3.5) {$k$};
\node at (0.5,-0.5) {$k+1'$};
\end{scriptsize}
\node at (5,1.5) {or};

\begin{scope}[shift={(8,0)}]
	\foreach \x in {-1,0,1,2,3}{
	\draw (\x,0)--(\x,1);
	\draw (\x,2)--(\x,3);}
\foreach \y in {0,1,2,3}{
	\draw (-2,\y)--(4,\y);}
\draw[dash pattern=on 2pt off 2pt] (-2.5,0.5)--(-1.5,0.5);
\draw[dash pattern=on 2pt off 2pt] (3.5,0.5)--(4.5,0.5);
\draw[dash pattern=on 2pt off 2pt] (-2.5,2.5)--(-1.5,2.5);
\draw[dash pattern=on 2pt off 2pt] (3.5,2.5)--(4.5,2.5);
\Ltoken{-1}{2}
\Rtoken{0}{0}
\Ltoken{2}{0}
\Ltoken{2}{2}
\begin{scriptsize}
\node at (-0.5,3.5) {$k$};
\node at (0.5,-0.5) {$k+1'$};
\end{scriptsize}
\end{scope}
\end{tikzpicture}

To the right of move 0
\end{center}

Left now has at most four additional non-copiable moves to which Right responds by playing in the integer. Since Right has to respond in the integer at most five times, he wins.
\end{proof}

%
%

\subsection{\snort on other Boards}
For \snort on a $2\times n$ empty grid, the following temperatures were found using CGSuite.

\begin{center}
\begin{tabular}{r|cccccc}
$n$&2&3&4&5&6&7\\\hline
$t$&$-1$&9/4&$-1$&5/2&$-1$&1
\end{tabular}
\end{center}

More generally, we make the following conjecture:
\begin{conjecture}\label{conj:SnortTemp}
The temperature of \snort on a board $B$ is at most the degree of $B$.
\end{conjecture}

Intuitively, this conjectures comes from the degree of $B$ being the maximum number of spaces one can `reserve' for themselves with a single move. There are cases in which the temperature is equal to the degree of the board. Here we need the idea of a \textit{universal vertex}, which is a vertex adjacent to all other vertices in the graph, \ie one for which the degree is $|V|-1$.
\begin{proposition}
Suppose $B$ is a graph with a universal vertex and $|V|=n+1$. For $G=(\Snort, B)$ we have $G=\pm n$, thus $t(G)=n$.
\end{proposition}
\begin{proof}
Let $v$ be a universal vertex of $B$. The good move for either player is to play on $v$, thus reserving all other vertices for themselves. All other possible moves will be dominated by this move. Thus $G=\{n\mid -n\}$.
\end{proof}


\section{Further work}





In \cref{thm:BPbound1} and all our applications we have bounded the confusion interval for all options. To improve these bounds, we will look at what the thermic version specifically is without having to build the whole thermograph, thus only having to bound the length of the confusion interval for the thermic options. Game professionals are good at this---identifying the most important options and discarding the others without going through a full analysis. This will be hard in general but are there hot games for which this is possible? 

Further, all bounds on the length of the vertical and oblique segments in the proof of \cref{thm:TempBound} are tight in certain cases, but are often much larger than the actual length of these segments. When restricting to specific classes of games it should be possible to improve these bounds and therefore the bound on the temperature itself.

\bibliographystyle{plain}
\bibliography{Biblio2019July}

\begin{thebibliography}{10}

\bibitem{Berlekamp1996}
Elwyn Berlekamp.
\newblock The economist's view of combinatorial games.
\newblock In {\em Games of no chance ({B}erkeley, {CA}, 1994)}, volume~29 of
  {\em Math. Sci. Res. Inst. Publ.}, pages 365--405. Cambridge Univ. Press,
  Cambridge, 1996.

\bibitem{Berlekamp2019}
Elwyn Berlekamp.
\newblock Temperatures of games and coupons.
\newblock In Urban Larsson, editor, {\em Games of No Chance 5}, volume~70 of
  {\em Mathematical Sciences Research Institute Publications}, pages 21--33.
  Cambridge University Press, 2019.

\bibitem{Berlekamp1988}
E.R. Berlekamp.
\newblock Blockbusting and domineering.
\newblock {\em J. Combin. Theory Ser. A}, 49(1):67--116, 1988.

\bibitem{BerlekampCG2004}
E.R. Berlekamp, J.H. Conway, and R.K. Guy.
\newblock {\em Winning ways for your mathematical plays. {V}ol. 1}.
\newblock A K Peters Ltd., Wellesley, MA, second edition, 2004.

\bibitem{Cazenave2015}
Tristan Cazenave.
\newblock Monte-{C}arlo approximation of temperature.
\newblock In {\em Games of no chance 4}, volume~63 of {\em Math. Sci. Res.
  Inst. Publ.}, pages 41--45. Cambridge Univ. Press, New York, 2015.

\bibitem{Conway2001}
J.~H. Conway.
\newblock {\em On numbers and games}.
\newblock A K Peters Ltd., Natick, MA, second edition, 2001.

\bibitem{DrummondCole20xx}
G.C. Drummond-Cole.
\newblock Temperature {$2$} in {D}omineering.
\newblock Preprint.

\bibitem{DrummondCole2005}
G.C. Drummond-Cole.
\newblock Positions of value {${}^*2$} in generalized domineering and chess.
\newblock {\em Integers}, 5(1):G6, 13, 2005.

\bibitem{Guy1996}
Richard~K. Guy.
\newblock Unsolved problems in combinatorial games.
\newblock In {\em Games of no chance ({B}erkeley, {CA}, 1994)}, volume~29 of
  {\em Math. Sci. Res. Inst. Publ.}, pages 475--491. Cambridge Univ. Press,
  Cambridge, 1996.

\bibitem{Kim1995}
Y.~Kim.
\newblock {\em New values in domineering and loopy games in {G}o}.
\newblock ProQuest LLC, Ann Arbor, MI, 1995.
\newblock Thesis (Ph.D.)--University of California, Berkeley.

\bibitem{Kim1996}
Y.~Kim.
\newblock New values in domineering.
\newblock {\em Theoret. Comput. Sci.}, 156(1-2):263--280, 1996.

\bibitem{LachmannMR2002}
Michael Lachmann, Cristopher Moore, and Ivan Rapaport.
\newblock Who wins {D}omineering on rectangular boards?
\newblock In {\em More games of no chance ({B}erkeley, {CA}, 2000)}, volume~42
  of {\em Math. Sci. Res. Inst. Publ.}, pages 307--315. Cambridge Univ. Press,
  Cambridge, 2002.

\bibitem{Mesdal2009}
G.~A. Mesdal.
\newblock Partizan splittles.
\newblock In M.~H. Albert and R.~J. Nowakowski, editors, {\em Games of No
  Chance 3}, number~56 in Math. Sci. Res. Inst. Publ., pages 447--461.
  Cambridge Univ. Press, 2009.

\bibitem{MullerES2004}
Martin M\"{u}ller, Markus Enzenberger, and Jonathan Schaeffer.
\newblock Temperature discovery search.
\newblock In {\em Proceedings of the 19th National Conference on Artifical
  Intelligence}, AAAI'04, pages 658--663. AAAI Press, 2004.

\bibitem{NowakowskiS2007}
R.~J. Nowakowski and A.~A. Siegel.
\newblock Partizan geography on {$K_n\times K_2$}.
\newblock In {\em Combinatorial number theory}, pages 389--401. de Gruyter,
  Berlin, 2007.

\bibitem{ShankarS2005}
Ajeet Shankar and Manu Sridharan.
\newblock New temperatures in {D}omineering.
\newblock {\em Integers}, 5(1):G4, 13, 2005.

\bibitem{Siegel2013}
Aaron~N. Siegel.
\newblock {\em Combinatorial game theory}, volume 146 of {\em Graduate Studies
  in Mathematics}.
\newblock American Mathematical Society, Providence, RI, 2013.

\bibitem{UiterwijkB2015}
Jos W. H.~M. Uiterwijk and Michael Barton.
\newblock New results for {D}omineering from combinatorial game theory endgame
  databases.
\newblock {\em Theoret. Comput. Sci.}, 592:72--86, 2015.

\bibitem{Wolfe1993}
David Wolfe.
\newblock Snakes in domineering games.
\newblock {\em Theoret. Comput. Sci.}, 119(2):323--329, 1993.

\end{thebibliography}

\end{document}